\documentclass[preprint,12pt]{elsarticle}

\usepackage{amssymb}
\usepackage{amsmath}

\usepackage{bm}
\usepackage{mathrsfs}
\usepackage{amsthm}

\newtheorem{thm}{Theorem}[section]

\newtheorem{cor}[thm]{Corollary}
\newtheorem{lem}[thm]{Lemma}

\theoremstyle{definition}
\newtheorem{de}[thm]{Definition}

\newtheorem{exa}[thm]{Example}
\newtheorem{re}[thm]{Remark}

\numberwithin{equation}{section}
\allowdisplaybreaks

\newcommand{\N}{\mathbb{N}}

\newcommand{\norm}[1]{\left\| #1 \right\|}

\newcommand{\Z}{\mathbb{Z}}

\journal{}

\begin{document}

\begin{frontmatter}



\title{Approximation properties and quantitative estimation for uniform ball-covering
property of operator spaces} 


\author[label1]{Qiyao Bao} 
\ead{qybao@mail.nankai.edu.cn}
\author[label1]{Rui Liu}
\ead{ruiliu@nankai.edu.cn}
\author[label1]{Jie Shen\corref{cor1}}
\cortext[cor1]{Corresponding author}
\ead{1710064@mail.nankai.edu.cn}
\affiliation[label1]{organization={School of Mathematical Sciences and LPMC
},
            addressline={Nankai University},
            city={Tianjin},
            postcode={300071},
            country={China}}

\tnotetext[label1]{This work was supported by the National Natural Science Foundation of China (Grant Nos. 11971348, 12071230 and 12471131).}

\begin{abstract}
In this paper, by dilation technique on Schauder frames, we extend Godefroy and Kalton's approximation theorem (1997), and obtain that a separable Banach space has the $\lambda$-unconditional bounded approximation property ($\lambda$-UBAP) if and only if, for any $\varepsilon>0$, it can be embeded into a $(\lambda+\varepsilon)$-complemented subspace of a Banach space with an $1$-unconditional finite-dimensional decomposition ($1$-UFDD). As applications on ball-covering property (BCP) (Cheng, 2006) of operator spaces, also based on the relationship between the $\lambda$-UBAP and block unconditional Schauder frames, we prove that if $X^\ast$, $Y$ are separable and (1) $X$ or $Y$ has the $\lambda$-reverse metric approximation property ($\lambda$-RMAP) for some $\lambda>1$; or (2) $X$ or $Y$ has an approximating sequence $\{S_n\}_{n=1}^\infty$ such that $\lim_n\|id-2S_n\| < 3/2$, then the space of bounded linear operators $B(X,Y)$ has the uniform ball-covering property (UBCP). Actually, we give uniformly quantitative estimation for the renormed spaces. We show that if $X^\ast$, $Y$ are separable and $X$ or $Y$ has the $(2-\varepsilon)$-UBAP for any $\varepsilon>0$, then for all $1-\varepsilon/2 < \alpha \leq 1$, the renormed space $Z_\alpha=(B(X,Y),\|\cdot\|_\alpha)$ has the $(2\alpha, 2\alpha+\varepsilon-2-\sigma)$-UBCP for all $0 <\sigma < 2\alpha+\varepsilon-2$. Furthermore, we point out the connections between the UBCP, u-ideals and the ball intersection property (BIP).
\end{abstract}




\end{frontmatter}



\section{Introduction}\label{sec1}

The approximation property which already appeared in Banach's book in 1932 \cite{B1932}, plays a fundamental role in the structure theory of Banach spaces. In 1955, Grothendieck \cite{G1955} initiated the first systematic study of the variants of the approximation property. At that time the main properties were the approximation property, the bounded approximation property and the basis property. Grothendieck's theorem states that for reflexive Banach spaces, the bounded approximation property is equivalent to the approximation property. In 1973, Enflo \cite{E1973} constructed a counterexample which shows that there exists a separable reflexive Banach space without the approximation property. Therefore, there exists a separable Banach space without a Schauder basis. In 1987, Szarek \cite{S1987} provided a very complex example, demonstrating that there exists a separable superreflexive Banach space with the bounded approximation property but without a Schauder basis. Now there is a ``chain" of approximation properties in between the standard ones \cite{C2001}, including the commuting bounded approximation property, the $\pi$-property and the finite-dimensional decomposition property and so on. There is a rich theory surrounding the approximation property for Banach spaces. What one concerns is whether a given Banach space with the weakest possible form of the approximation property actually has the strongest possible form. For example, Pełczyński \cite{P1971} and independently Johnson, Rosenthal and Zippin \cite{JRZ1971} proved that a separable Banach space has the bounded approximation property if and only if it is a complemented subspace of a Banach space with a Schauder basis.

Frames for Banach spaces have important relations with various forms of the Banach space approximation properties and the dilation theory of Banach spaces \cite{BHLS2025,ADV2024,MAMS2014,JFA2014,JFA2018}. In \cite{CHL}, Casazza, Han and Larson generalized the concepts of Hilbert frames and Schauder bases using dilation techniques and introduced Schauder frames in separable Banach spaces. They proved that a separable Banach space has a Schauder frame if and only if it has the bounded approximation property. There are plenty of investigations concerning (unconditional) Schauder frames \cite{BF2024,BFL2015,ST2008,ST2014,LT2025,L2010,LZ2010,OSS2011} which are effective tool for the study of Banach space theory. Based on the connection between the bounded approximation property and Schauder frames, Beanland et al. \cite{BFL2015} showed that if a reflexive Banach space has a Schauder frame, then it can be complementably embedded into a reflexive space with a Schauder basis. Moreover, Liu and Ruan \cite{LR2016} introduced the concept of completely bounded frames for operator spaces, and proved that a separable operator space has a completely bounded frame if and only if it has the completely bounded approximation property if and only if it is completely isomorphic to a completely complemented subspace of an operator space with a completely bounded basis. In this paper, we present the relationship between the $\lambda$-unconditional bounded approximation property ($\lambda$-UBAP) and block unconditional Schauder frames as follows.

\begin{lem}
  Let $X$ be a separable Banach space with the $\lambda$-UBAP and $\varepsilon > 0$. Then there exists a (block unconditional) Schauder frame
  $\{(x_n,f_n)\}_{n\in\mathbb{N}+} \subseteq X\times X^\ast$ such that for all $x\in X$ we have $$x=\sum_{n=1}^{\infty}f_n(x) x_n.$$ And there exists a subsequence
  $\{N_k\}_{k=1}^{\infty}$ of $\N_+$ such that
  $$\sup_{N\in\mathbb{N}_+,\theta_k=\pm1}\left\|\sum_{k=1}^{N}\theta_k \sum_{n=N_k+1}^{N_{k+1}}f_n \otimes x_n\right\|\leq \lambda+\varepsilon  \text{ and } \sup_{N\in\mathbb{N}_+}\left\|\sum_{n=1}^{N}f_n \otimes x_n\right\|\leq \lambda+\varepsilon.$$
  In addition, we can assume $\|x_n\| =1$ and $\|f_n\| \leq 1$  for all $n \in \N_+$.
\end{lem}

Based on the relationship between the $\lambda$-UBAP and block unconditional Schauder frames, we obtain the following extended version of Godefroy and Kalton's approximation theorem \cite{GK1997} and prove it by the dilation technique on Schauder frames in Section 2.

\begin{thm}
Let $X$ be a separable Banach space. Then $X$ has the $\lambda$-UBAP if and only if for any $\varepsilon>0$, $X$ embeds as a $(\lambda+\varepsilon)$-complemented subspace of a Banach space $Y_{\varepsilon}$ with an $1$-unconditional finite-dimensional decomposition.
\end{thm}

Moreover, using the relationship between the $\lambda$-UBAP and block unconditional Schauder frames, we discuss the ball-covering property of operator spaces and give uniformly quantitative estimations as applications.

The notion of the ball-covering property was introduced by Cheng \cite{C2006}.
A normed space $X$ is said to have the {\it ball-covering property (BCP)} if its unit sphere can be covered by countably many closed or open balls off the origin. The centers of the balls are called the {\it BCP points} of $X$. Assume that a collection of countable balls $\{B(x_{n},r_{n})\}_{n=1}^{\infty}$ forms a ball-covering of $X$, that is, the unit sphere of $X$ can be covered by $\bigcup_{n=1}^{\infty} B(x_{n},r_{n})$ with $\|x_n\|>r_n$ for all $n\in\mathbb{N}_+$. Then $X$ is said to have the {\it $r$-strong ball-covering property ($r$-SBCP)} \cite{LZ2021} if there exists a positive number $r>0$ such that $\sup r_n\leq r$. Moreover, $X$ is said to have the {\it $(r,\delta)$-uniform ball-covering property (($r,\delta$)-UBCP)} \cite{LZ2021}, if $X$ has the $r$-SBCP, and there exists another positive number $\delta>0$ such that $B(x_{n},r_{n}) \cap B(0,\delta) = \emptyset$ for all $n\in\mathbb{N}_+$. The normed space $X$ is said to have the SBCP (UBCP, respectively) if there exists $r>0$ ($r>0$ and $\delta>0$, respectively) such that $X$ has the $r$-SBCP (($r,\delta$)-UBCP, respectively). The BCP has deep relationships with many important properties of Banach spaces, such as separability, completeness, reflexivity, smoothness, Radon-Nikodym property \cite{CWZ2011}, uniform convexity, uniform non-squareness \cite{CLL2010}, strict convexity and dentability \cite{SC2015,SC2018}, and universal finite representability and B-convexity \cite{Z2012}. The BCP also plays an important role in the study of geometric and topological properties of Banach spaces \cite{ AG2023,CKZ2020,CLL2023,FR2016,LZ2020,S2021}.

It is obvious from the definition of the BCP that separable normed spaces have the BCP, but the converse is not true \cite{C2006,CCL2008}. Cheng \cite{C2006} proved that the non-separable space $\ell^{\infty}$ has the BCP. However, Cheng, Cheng and Liu \cite{CCL2008} showed that $\ell^{\infty}$ can be renormed such that the renormed space fails the BCP. The class of non-separable Banach spaces with $w^\ast$-separable dual have very rich geometrical structure. Cheng, Shi and Zhang \cite{CSZ2009} proved that the dual space $X^\ast$ is $w^\ast$-separable if and only if $X$ can be $(1+\varepsilon)$-renormed to have the BCP for any $\varepsilon>0$. Furthermore, Fonf and Zanco \cite{FZ32009} proved that $X^\ast$ is $w^\ast$-separable if and only if $X$ can be $(1+\varepsilon)$-renormed to have the SBCP for any $\varepsilon>0$. Liu et al. \cite{LLLZ2022} investigated the BCP on non-commutative spaces of operators $B(X,Y)$ and proved that if $X$ is a Banach space with $X^\ast$ separable, then any subspace $E$ of $B(X,\ell_{p})$ containing the space of finite rank operators $F(X,\ell_{p})$ has the UBCP for $1<p<\infty$. Moreover, they presented some necessary conditions for $B(X,Y)$ to have the BCP. In \cite{BLS2025}, we given a characterization for the BCP of the renormed space $X_\alpha=(B(X),\|\cdot\|_\alpha)$ $(0\leq\alpha\leq1)$, where $X$ is a Banach space with a shrinking $1$-unconditional basis. In this paper, we focus on the Banach spaces with various classical approximation properties and present some sufficient conditions for $B(X,Y)$ to have the UBCP. We also give a quantitative estimation for the UBCP of the renormed spaces.

\begin{thm}
  Let $X$ and $Y$ be Banach spaces with $X^\ast$ and $Y$ separable. If $X$ or $Y$ has the $(2-\varepsilon)$-UBAP for any $\varepsilon>0$, then for all $1-\varepsilon/2 < \alpha \leq 1$, the renormed space $Z_\alpha=\left(B(X,Y), \|\cdot\|_\alpha \right)$ has the $(2\alpha,2\alpha+\varepsilon-2-\sigma)$-UBCP for all $0<\sigma<2\alpha+\varepsilon-2$.
\end{thm}

\begin{cor}
  Let $X$ and $Y$ be Banach spaces with $X^\ast$ and $Y$ separable. If $X$ or $Y$ has the the $\lambda$-RMAP for some $\lambda>1$, then $B(X,Y)$ has the UBCP.
\end{cor}

\begin{cor}
     Let $X$ and $Y$ be Banach spaces with $X^\ast$ and $Y$ separable.  If $X$ or $Y$ has an approximating sequence $\{S_n\}_{n=1}^\infty$ such that $\lim_n\|id-2S_n\| < 3/2$, then  $B(X,Y)$ has the UBCP.
\end{cor}

Finally, based on the equivalent relationships between the unconditional metric approximation property, u-ideals and the ball intersection property, we point out the connections between the UBCP, u-ideals and the ball intersection property.

The following is a list of notations that will be used in this article.

 \begin{itemize}
    \item $\N_+$ --- the set of positive integers.
    \item $B(X,Y)$ --- the space of bounded linear operators from Banach space $X$ into Banach space $Y$.
    \item $\operatorname{span} \{E\} $ --- the linear space spanned by the set $E$.
    \item $\operatorname{dim} (E) $ --- the dimension of $E$ as a linear space.
    \item $a \otimes b$ --- the rank one operator defined by $(a\otimes b) (x) = a(x)b$ for all $x,b \in X$ and $a \in X^*$.
    \item $B_X$ and $S_X$ --- the closed unit ball and unit sphere of Banach space $X$, respectively.
    \item $B(x,r)$ --- the closed ball with center $x$ and radius $r$.
    \item $id_X$ --- the identity operator on Banach space $X$.
\end{itemize}

\section{Approximation properties and Schauder frames}


Firstly, we introduce some basic definitions and results concerning the approximation property \cite{C2001} as follows.

\begin{de}
A Banach space $X$ is said to have the approximation property if for every compact set $K$ in $X$ and every $\varepsilon>0$, there is a finite rank operator $S$ such that $\|Sx-x\|\leq\varepsilon$ for every $x\in K$.
\end{de}

\begin{de}
Let $X$ be a Banach space and $1\leq\lambda<\infty$. We say that $X$ has the $\lambda$-bounded approximation property ($\lambda$-BAP for short) if for every $\varepsilon >0$ and every compact set $K$ in $X$, there is a finite rank operator $S$ such that $\|S\|\leq \lambda$ and $\|Sx-x\|\leq\varepsilon$ for every $x\in K$. We say that $X$ has the bounded approximation property (BAP for short) if $X$ has the $\lambda$-BAP for some $\lambda$. Moreover, a Banach space is said to have the metric approximation property (MAP for short) if it has the $1$-BAP.
\end{de}


The following lemma provides an alternative formulation of the BAP, and we omit its proof (c.f. \cite{C2001}).

\begin{lem}\label{lem1}
Let $X$ be a separable Banach space. Then $X$ has the $\lambda$-BAP if and only if there is a sequence of finite rank operators $\{S_n\}_{n=1}^\infty$ on $X$ converging strongly to the identity operator such that $S_m S_n=S_n$ for all $n<m$, and $\lim\sup_n\left\|S_n\right\| \leq \lambda$.
\end{lem}

We call such a sequence $\{S_n\}_{n=1}^\infty$ in Lemma \ref{lem1} an approximating sequence. Moreover, if $\{S_n\}_{n=1}^\infty$ is a commuting approximation sequence, that is, $S_n S_m=S_m S_n$ for all $n, m \in \N_+$, then $X$ is said to have the $\lambda$-commuting bounded approximation property ($\lambda$-CBAP for short).

\begin{de}\label{UMAP}
  Let $X$ be a separable Banach space, then $X$ has the $\lambda$-unconditional BAP ($\lambda$-UBAP for short) if for every $\varepsilon>0$ there is an approximating sequence $\{S_n\}_{n=1}^\infty$ such that if $A_n=S_n-S_{n-1}$ (and $S_0=0)$ for $n\in\mathbb{N}_+$ then for every $N\in\mathbb{N}_+$, we have
  $$\sup_{N \in \N_+, \theta_i= \pm 1}\left\|\sum_{i=1}^N \theta_i A_i\right\| \leq \lambda+\varepsilon.$$
  A Banach space is said to have the unconditional MAP (UMAP for short) if it has the $1$-UBAP.
\end{de}

\begin{lem}\label{lem3.4}
  Let $\{A_n\}_{n=1}^\infty$ be an operator sequence in Definition \ref{UMAP} of a separable Banach space $X$, then for all $\{a_n\}_{n=1}^\infty \in \ell^\infty$, the following inequality holds
  \[ \norm{\sum_{n=1}^\infty a_n A_n} \leq \|\{a_n\}\|_\infty (\lambda+\varepsilon).\]
  Moreover, if $X$ has the $\lambda$-UBAP, then there is an approximating sequence $\{S_n\}_{n=1}^\infty$ such that \[\limsup_{n}\|id_X -2 S_n\| \leq \lambda+\varepsilon.\]
\end{lem}
\begin{proof}
   For all $N \in \N_+$, $x \in X$ and $f \in X^*$, we have
   \begin{align*}
       &\left| f \left(\sum_{n=1}^{N}a_n  A_n x\right)\right|  \leq \sum_{n=1}^{N}|a_n| \left| f(A_n x) \right|\\
       \leq &\sup_{n\leq N}|a_n|\sum_{n=1}^{N} |f(A_n x)| \leq (\lambda+\varepsilon)\|\{a_n\}\|_\infty \|f\|\|x\|.
   \end{align*}
  Taking the supremum over all $N \in \N_+$, $x \in B_X$ and $f \in B_{X^*}$, then we have $$\norm{\sum_{n=1}^\infty a_n A_n} \leq \|\{a_n\}\|_\infty (\lambda+\varepsilon).$$
  Note that for all $\varepsilon>0$, by the uniform boundedness principle, we have \[\|id_X -2 S_n\| \leq \liminf_m \norm{-\sum_{1\leq j \leq n}A_j+\sum_{n<j\leq m}A_j} \leq \lambda+\varepsilon.\]
  Thus there is an approximating sequence $\{S_n\}_{n=1}^\infty$ such that $\limsup_{n}\|id_X -2 S_n\| \leq \lambda+\varepsilon$.
\end{proof}

An unconditional approximating sequence is actually an approximating sequence and it is worthwhile to know the relationship between the BAP constant and the UBAP constant.


\begin{re}\label{rem3.5}
  Let $X$ be a separable Banach space. For a given unconditional basis $\{x_n\}$, denote the basis constant and the unconditional constant of $\{x_n\}$ by $\textbf{bc}(\{x_n\})$ and $\textbf{ubc}(\{x_n\})$, respectively. Let
  $$\textbf{bc}(X)=\inf\left\{\textbf{bc}(\{x_n\}): \{x_n\} \ \text{a basis of}\ X\right\}$$ and  $$\textbf{ubc}(X)=\inf\left\{\textbf{ubc}(\{x_n\}): \{x_n\}\ \text{an unconditional basis of}\ X \right\}.$$
  It was proved in \cite{M1998} that \[1\leq \textbf{bc}(X) \leq  \frac{1+\textbf{ubc}(X)}{2}.\]
  Similarly, let $$\textbf{BAP}(X)=\inf\left\{\lambda: \ X \ \text{has the}\ \lambda\text{-BAP}\right\}$$
  and $$\textbf{UBAP}(X)=\inf\left\{\lambda: \ X \ \text{has the}\ \lambda\text{-UBAP}\right\},$$
  by Lemma \ref{lem3.4}, we have \[1\leq \textbf{BAP}(X) \leq  \frac{1+\textbf{UBAP}(X)}{2}.\]
\end{re}

The existence of an approximating sequence $\{S_n\}_{n=1}^\infty$ with $\lim_{n}\|id_X -2 S_n\| = \lambda$ is not enough to determine whether $X$ has the $\lambda$-UBAP or even the UBAP since $L^1[0,1]$ has the MAP but fails the UBAP. However, when $\lambda=1$, Casazza and Kalton \cite{C1991} proved that $X$ has the UMAP.

\begin{thm}
Let $X$ be a separable Banach space. Then $X$ has the UMAP if and only if there is an approximating sequence $\left\{S_n\right\}_{n=1}^\infty$ such that $$\lim_{n \rightarrow \infty}\left\|id_X-2 S_n\right\|=1.$$
\end{thm}

It is known that Banach space frames have close relationship with various forms of the approximation properties. Casazza, Han and Larson \cite{CHL} introduced the concept of Schauder frames for Banach spaces, which is a natural generalization of frames for Hilbert spaces and Schauder bases for Banach spaces. The techniques inspired by Schauder frames \cite{ST2008,L2010} might play a critical role in the study of Banach space geometry and operator algebras \cite{LR2016}.

 \begin{de}
  Let $X$ be a separable Banach space and $\left\{(x_{n},f_{n})\right\}_{n=1}^\infty \subseteq X \times X ^{\ast}$. Then $\left\{(x_{n},f_{n})\right\}_{n=1}^\infty$ is a Schauder frame of $X$ if for all $ x \in X$, we have \[x=\sum_{n=1}^{\infty} f_{n}(x) x_{n}.\]
  \end{de}

  \begin{de}
 Let $X$ be a separable Banach space and $\left\{(x_{n},f_{n})\right\}_{n=1}^\infty\subseteq X \times X ^{\ast}$ be a Schauder frame. Then $\left\{(x_{n},f_{n})\right\}_{n=1}^\infty$ is said to be block unconditional if there exists a subsequence $\{N_k\}_{k=1}^\infty$ of $\mathbb{N}_+$ such that
  \begin{equation}\label{bufb}
  \sup_{N\in\mathbb{N_+},\theta_k=\pm1}\left\|\sum_{k=1}^{N}\theta_k \sum_{n=N_k+1}^{N_{k+1}}f_n\otimes x_n\right\| < \infty.
  \end{equation}
 We call $\sup_{N\in\mathbb{N_+},\theta_k=\pm1}\left\|\sum_{k=1}^{N}\theta_k \sum_{n=N_k+1}^{N_{k+1}}f_n\otimes x_n\right\|$ satisfying (\ref{bufb}) the block unconditional frame bound of $\left\{(x_{n},f_{n})\right\}_{n=1}^\infty$.
\end{de}

\begin{lem}\label{decomposable}
  Let $X$ be a separable Banach space with the $\lambda$-UBAP and $\varepsilon > 0$. Then there exists a (block unconditional) Schauder frame $\{(x_n,f_n)\}_{n\in\mathbb{N}+} \subseteq X\times X^\ast$ such that for all $x\in X$ we have
  $$x=\sum_{n=1}^{\infty}f_n(x) x_n.$$  And there exists a subsequence
  $\{N_k\}_{k=1}^{\infty}$ of $\N_+$ such that
  $$\sup_{N\in\mathbb{N}_+,\theta_k=\pm1}\left\|\sum_{k=1}^{N}\theta_k \sum_{n=N_k+1}^{N_{k+1}}f_n \otimes x_n\right\|\leq \lambda+\varepsilon  \text{ and } \sup_{N\in\mathbb{N}_+}\left\|\sum_{n=1}^{N}f_n \otimes x_n\right\|\leq \lambda+\varepsilon.$$
  In addition, we can assume $\|x_n\| = 1$ and $\|f_n\| \leq 1$  for all $n \in \N_+$.
\end{lem}
\begin{proof}
  Since $X$ has the $\lambda$-UBAP, by Definition \ref{UMAP} and Remark \ref{rem3.5}, we can take an unconditional approximating sequence $\left\{S_n\right\}_{n=1}^\infty$ on $X$ converging strongly to the identity operator such that $S_m S_n=S_n$ for all $n<m$. By passing to a subsequence, we can assume $\sup_n \left\|T_n\right\| \leq \lambda+2^{-1}\varepsilon$.

  Let $A_n=S_n-S_{n-1}$ (with $S_0=0$), $E_n=A_n X$ and $d_n=\dim(E_n)$ for every $n\in\mathbb{N}_+$, then $d_n<\infty$ and for all $x \in X$, $x=\sum_{n=1}^{\infty} A_n(x)$. By Definition \ref{UMAP}, for every $\varepsilon>0$, we have
  $$\sup_{N \in \N_+, \theta_i= \pm 1}\left\|\sum_{i=1}^N \theta_i A_i\right\| \leq \lambda+ 2^{-1}\varepsilon.$$

  For every finite-dimensional space $E_{n}$, let  $\left\{z_1^n,z_2^n,\dots,z_{d_{n}}^n\right\}$ be a Schauder basis of $E_{n}$ and its normalized bi-orthogonal basis be $$\left\{(z_1^n,z_1^{n*}),(z_2^n,z_2^{n*}),\dots,(z_{d_{n}}^n,z_{d_{n}}^{n*})\right\}.$$
  Then for all $x \in X$, we have
  $$x=\sum_{n=1}^{\infty}A_{n}(x)=
  \sum_{n=1}^{\infty}\sum_{j=1}^{d_{n}}z_j^{n*}(A_{n}(x))z_j^{n}.$$

  Let $C_n$ be the basis constant for $\left\{(z_1^n,z_1^{n*}),(z_2^n,z_2^{n*}),\dots,(z_{d_{n}}^n,z_{d_{n}}^{n*})\right\}$, then $C_n<\infty$.
  Choose $M_n \in \N_+$  such that \[M_n \geq  2\lambda+\varepsilon \quad\text{ and }\quad \frac{C_{n}}{M_{n}}\leq\frac{\varepsilon}{4\lambda+2\varepsilon}.\]
  Let $$x_{i}=z_{j}^{n} \quad \text{ and } \quad f_{i}=\frac{z_{j}^{n *}\circ A_n}{M_{n}} \quad \text{if} \quad i=\sum_{k=1}^{n-1} M_{k} d_{k} + rd_{n}+j,$$ where $r=0,\dots, M_{n}-1$ and $j=1,\dots, d_{n}$. It is obvious that $\|x_k\|= 1$ and $\|f_k\|\leq 1$ for every $k\in \N_+$. Thus for all $x \in X$, we have
  \begin{align*}
    x & =\sum_{n=1}^{\infty}\sum_{j=1}^{d_{n}}z_j^{n*}(A_{n}(x))z_j^{n} \\
      & =\sum_{n=1}^{\infty}M_{n}\sum_{j=1}^{d_{n}}\frac{z_j^{n*}(A_{n}(x))z_j^{n}}{M_{n}} \\
      & =\sum_{n=1}^{\infty} f_{n}(x)x_{n}\\
      & =\sum_{n=1}^{\infty} (f_{n}\otimes x_{n})(x).
  \end{align*}
  Therefore $\sum_{n=1}^{\infty} f_{n}\otimes x_{n}$ converges strongly to $id_X$.

  Next we estimate the frame bound and block unconditional frame bound of
  $\{(x_n,f_n)\}_{n\in\mathbb{N}+}$. For every well-defined $s=\sum_{k=1}^{n-1} M_{k} d_{k} + rd_{n}+j$ and every $x \in X$, we have

  \begin{align*}
    \left\|\sum\limits_{k=1}^{s} f_{k}(x) x_{k}\right\| &=\left\|\sum\limits_{k=1}^{n-1}A_k(x)+\frac{r}{M_{n}}A_n(x)+
    \sum\limits_{k=1}^{j}\frac{z_k^{n*}(A_{n}(x)) z_k^{n}}{M_{n}}\right\|\\
    &=\left\|\sum\limits_{k=1}^{n-1}A_k(x)+\frac{r}{M_{n}}\left(\sum\limits_{k=1}^n A_k(x)-\sum\limits_{k=1}^{n-1} A_k(x)\right)+
    \sum\limits_{k=1}^{j}\frac{z_k^{n*}(A_{n}(x)) z_k^{n}}{M_{n}}\right\|\\
    &=\left\|\frac{r}{M_{n}}\sum\limits_{k=1}^{n}A_k(x)+\left(1-\frac{r}{M_{n}}\right)
    \sum\limits_{k=1}^{n-1} A_k(x)+
    \sum\limits_{k=1}^{j}\frac{z_k^{n*}(A_{n}(x)) z_k^{n}}{M_{n}}\right\|\\
    &=\left\|\frac{r}{M_{n}}S_n(x)+\left(1-\frac{r}{M_{n}}\right)S_{n-1}(x)+
    \sum\limits_{k=1}^{j}\frac{z_k^{n*}(A_{n}(x)) z_k^{n}}{M_{n}}\right\|\\
    & \leq\frac{r}{M_{n}}\|S_n(x)\|+\left(1-\frac{r}{M_{n}}\right)\|S_{n-1}(x)\|
    +\frac{\left\|\sum\limits_{k=1}^{j} z_k^{n*}(A_{n}(x)) z_k^{n}\right\|}{M_{n}}\\
    & \leq\frac{r}{M_{n}}\left(\lambda+\frac{1}{2}\varepsilon\right)\|x\|+
    \left(1-\frac{r}{M_{n}}\right)\left(\lambda+\frac{1}{2}\varepsilon\right)\|x\|\\
    &\quad+\frac{\left\|\sum\limits_{k=1}^{j} z_k^{n*}(A_{n}(x)) z_k^{n}\right\|}{M_{n}}\\
    & \leq \left(\lambda+\frac{1}{2}\varepsilon\right)\|x\|+\frac{C_{n}}{M_{n}}\|A_{n}(x)\|\\
    & \leq \left(\lambda+\frac{1}{2}\varepsilon\right)\|x\|+\frac{\varepsilon}{4\lambda+2\varepsilon}\|A_{n}(x)\|\\
    & \leq \left(\lambda+\frac{1}{2}\varepsilon\right)\|x\|+\frac{\varepsilon}{4\lambda+2\varepsilon}\left(\|S_n x\|+\|S_{n-1}x\|\right)\\
    & \leq  \left(\lambda+\frac{1}{2}\varepsilon\right)\|x\|+\frac{\varepsilon}{4\lambda+2\varepsilon}
    \left(2\lambda+\varepsilon\right)\|x\|\\
    &= \left(\lambda+\varepsilon\right)\|x\|.
    \end{align*}
    Therefore $$\sup_{N\in\mathbb{N}_+}\left\|\sum_{n=1}^{N}f_n \otimes x_n\right\|\leq
    \lambda+\varepsilon.$$

    Let $N_1=0$ and for all $k \in \N_+$ with $k \geq 2$, let \[N_k=\sum_{i=1}^{k-1} d_{i}M_{i}, \]
    then $A_n=\sum_{n=N_k+1}^{N_{k+1}} f_{n}\otimes x_{n}$.
    Therefore, we have
    $$\sup_{N\in\mathbb{N}_+,\theta_k=\pm1}\left\|\sum_{k=1}^{N}\theta_k \sum_{n=N_k+1}^{N_{k+1}}f_n \otimes x_n\right\|=\sup_{N\in\mathbb{N}_+,\theta_k=\pm1}\left\|\sum_{k=1}^{N}\theta_k A_k\right\|\leq \lambda+\varepsilon. $$
  \end{proof}

\begin{de}(\cite{LSZ2023})
Let $X$ be a separable Banach space and $\left\{(x_{n},f_{n})\right\}_{n\in\N_+} \subseteq X\times X^*$ be a Schauder frame of $X$. Then $\left\{(x_{n},f_{n})\right\}_{n\in\N_+}$ is an unconditional Schauder frame if \[x=\sum_{n=1}^{\infty} f_{n}(x)x_n\] converges unconditionally for all $x \in X$.
\end{de}

If $X$ has an (unconditional) Schauder frame, then it is straightforward to construct
a separable Banach space $Z$ with an (unconditional) basis such that $X$ is complemented in $Z$. But if $X$ only has a block unconditional Schauder frame, that is, $X$ has the UBAP, then $X$ is complemented in a separable Banach space with an unconditional finite-dimensional decomposition.

\begin{de}
  A sequence $\{E_n\}_{n=1}^\infty$ of finite-dimensional subspaces of a Banach space $X$ is called an unconditional finite-dimensional decomposition for $X$ (UFDD for short) if for every $x\in X$ there is a unique sequence $x_n\in E_n$ such that $x=\sum_{n}x_n$ and this series converges unconditionally. In this case, we will write $X=\sum_{n}\oplus E_n$ and say $X$ has an unconditional finite-dimensional decomposition. We say $X$ has the $\lambda$-UFDD, if for all $\varepsilon>0$, there is a sequence $\{E_n\}_{n=1}^\infty$ such that
  \[\sup_{\varepsilon_n=\pm1}\norm{\sum_{n}\varepsilon_n x_n} \leq (\lambda+\varepsilon) \norm {\sum_{n} x_n}.\]
\end{de}

In general, it is very difficult to discover whether a concrete Banach space such as the Orlicz space or the Lorentz space which is complemented in a Banach space with an unconditional basis has an UFDD \cite{C2001}.

The following theorem is due to Godefroy and Kalton's approximation theorem \cite{GK1997}. Here we give an extended version and a new proof by using dilation techniques on Schauder frames.

\begin{thm}
Let $X$ be a separable Banach space. Then $X$ has the $\lambda$-UBAP if and only if for any $\varepsilon>0$, $X$ embeds as a $(\lambda+\varepsilon)$-complemented subspace of a Banach space $Y_{\varepsilon}$ with an $1$-UFDD.
\end{thm}
\begin{proof}
The ``if" part is trivial, we only need to prove the ``only if" part.

Since $X$ has the $\lambda$-UBAP, by Lemma \ref{decomposable}, for every $\varepsilon > 0$, there exists a block unconditional Schauder frame $\{(x_n,f_n)\}_{n\in\mathbb{N}_+} \subseteq X\times X^\ast$ such that
$x=\sum_{n=1}^{\infty}f_n(x) x_n$ for all $x\in X$ and there exists a subsequence
 $\{N_k\}_{k=1}^{\infty}$ of $\N_+$ such that $$\sup_{N\in \N_+,\theta_k=\pm1}\left\|\sum_{k=1}^{N}\theta_k \sum_{n=N_k+1}^{N_{k+1}}f_n\otimes x_n\right\|\leq \lambda+\varepsilon \text{ and } \sup_{N\in\mathbb{N}_+}\left\|\sum_{n=1}^{N}f_n \otimes x_n\right\|\leq \lambda+\varepsilon.$$
Let $\{I_k:=\{i\}_{i=N_k+1}^{N_{k+1}}\}_{k\in\mathbb{N}_+}$ and it is a partition of $\mathbb{N}_+$. Let $\{z_k\}_{k\in\mathbb{N}_+}$ be a copy of $\{x_k\}_{k\in\mathbb{N}_+}$ and $Y_k=\overline{\mathrm{span}}\{z_n: n\in I_k\}$, then $Y_k$ is a finite-dimensional subspace since every $I_k$ is finite. Let $c_{00}$ be the collection of all sequences of real numbers with all but finite nonzero elements. For every $k=1,2,\cdots$, let $\{a_{k}\}_{k\in \N_+}\in c_{00}$, we define the space $Y_\varepsilon$ as the completion of $(\sum_{k}\oplus Y_k)$ under the norm
$$\left\|\sum_{k=1}^{\infty}a_{k}z_k\right\|_{Y_{\varepsilon}}=\max\left\{\sup_{N \in \N_+, \theta_n=\pm1}\left\|\sum_{n=1}^{N}\theta_n\sum_{k\in I_n}a_{k}x_k\right\|,\sup_{N\in \N_+}\norm{\sum_{k=1}^N a_k x_k}\right\}.$$
It is easily seen that $\textbf{bc}(Y_k)=1$ for all $k$ and follows immediately that $Y_\varepsilon$ is a Banach space and $\{Y_k\}_{k=1}^{\infty}$ is an $1$-UFDD of $Y_{\varepsilon}$.

Define the map $S:Y_{\varepsilon}\rightarrow X$ by $$S(z)=\sum_{k=1}^{\infty}a_{k}x_k$$ for all $z=\sum_{k=1}^{\infty}a_{k}z_k\in Y_{\varepsilon}$, then \[\|Sz\|=\norm{\sum_{k=1}^{\infty}a_{k}x_k}\leq \sup_{N\in \N_+}\norm{\sum_{k=1}^N a_k x_k}\leq \|z\|.\]
Thus $\|S\| \leq 1$ and $S$ is a bounded linear operator.

Define the map $T:X\rightarrow Y_{\varepsilon}$ by $$T(x)=\sum_{k=1}^{\infty}\sum_{n\in I_k}f_n(x)z_n$$ for all $x=\sum_{n=1}^{\infty}f_n(x)x_n\in X$, then
\[\begin{aligned}
  \|Tx\|&=\norm{\sum_{k=1}^{\infty}\sum_{n\in I_k}f_n(x)z_n}\\
  &=\max\left\{\sup_{N \in \N_+, \varepsilon_n=\pm1}\left\|\sum_{n=1}^{N}\varepsilon_n\sum_{k\in I_n}f_k(x)x_k\right\|,\sup_{N\in \N_+}\norm{\sum_{k=1}^N f_k(x) x_k}\right\}\\
  &\leq (\lambda+\varepsilon)\|x\|.
\end{aligned}\]
Thus $\|T\| \leq \lambda+\varepsilon$ and $T$ is a bounded linear operator.

For every $x\in X$, we have
\begin{align*}
  & S\circ T(x)=S \circ T\left(\sum_{n=1}^\infty f_n(x)x_n\right)\\
= & S\left(\sum_{k=1}^{\infty}\sum_{n\in I_k}f_n(x)z_n\right)=\sum_{k=1}^{\infty}\sum_{n\in I_k}f_n(x)x_n=x.
\end{align*}
Thus $S\circ T=id_X$.

It is clear that $T$ is injective and $S$ is surjective. So $T(X)$ is a closed subspace of $Y_\varepsilon$. Let $P=T \circ S$, then
$$P^2= (T \circ S) \circ(T \circ S)=T \circ S=P.$$ Since $S$ is surjective,  we obtain $P: Y_{\varepsilon} \rightarrow T(X)$ is a projection with $\|P\|\leq \|S\|\|T\| \leq \lambda+\varepsilon$.
Therefore $T(X)$ is a complemented subspace of $Y_{\varepsilon}$. Then $X$ is isomorphic to a $(\lambda+\varepsilon)$-complemented subspace $T(X)$ of $Y_{\varepsilon}$ with an 1-UFDD.
\end{proof}

\section{The UBCP of operator spaces}

Let $X$ and $Y$ be Banach spaces, and let $K(X,Y)$ be the ideal of compact operators in $B(X,Y)$. We fix a real number $\alpha\in [0,1]$, and define a new norm by  \begin{align}\|\cdot\|_{\alpha}=\alpha\|\cdot\|_{B(X,Y)}+(1-\alpha)\|\cdot\|_{B(X,Y)/ K(X,Y)}. \label{renorm}\end{align}
Denote $B(X,Y)/ K(X,Y)$ by $Q(X,Y)$. Then we consider the ball-covering property of the renormed space $Z_\alpha=(B(X,Y),\|\cdot\|_{\alpha})$.

\begin{thm}\label{thm3.12}
Let $X$ be a Banach space with $X^\ast$ separable and $Y$ be a separable Banach space with the $(2-\varepsilon)$-UBAP for any $\varepsilon>0$. Then $B(X,Y)$ has the $(2,\varepsilon-\sigma)$-UBCP for all $0<\sigma<\varepsilon$.
\end{thm}
\begin{proof}
By Lemma \ref{decomposable}, for all $0<\varepsilon_1<\sigma/2$, there exists a block unconditional Schauder frame $\{(y_n,f_n)\}_{n\in\mathbb{N}} \subseteq Y\times Y^\ast$ such that
$$y=\sum_{n=1}^{\infty}f_n(y) y_n=\sum_{n=1}^{\infty}(f_n\otimes y_n)(y)$$ for all $y\in Y$ and there exists a subsequence $\{N_k\}_{k=1}^{\infty}$ of $\N_+$ such that $$\sup_{N\in\N_+,\theta_k=\pm1}\left\|\sum_{k=1}^{N}\theta_k \sum_{n=N_k+1}^{N_{k+1}}f_n\otimes y_n\right\|\leq 2-\varepsilon+\varepsilon_1  \text{ and } \sup_{N\in\mathbb{N}_+}\left\|\sum_{n=1}^{N}f_n \otimes y_n\right\|\leq
2-\varepsilon+\varepsilon_1.$$

For all $T \in B(X,Y)$ with $\|T\|=1$ and for all $x \in X$, we have
  \[Tx=\sum_{n=1}^{\infty}\left(f_n T\otimes y_n\right) (x).\]
Now we show that for all $0<\varepsilon_2< \sigma/2$, there exists $k_0$ such that
 \[\max\left\{1-\frac{\varepsilon_2}{16},\frac{3}{4}\right\}< \norm{\sum_{n=1}^{N_{k_0}}f_nT\otimes y_n} \leq (2-\varepsilon+\varepsilon_1)\|T\|=2-\varepsilon+\varepsilon_1.\] We only need to show that there exists $k_0$ such that $$\max\left\{1-\frac{\varepsilon_2}{16},\frac{3}{4}\right\}< \norm{\sum_{n=1}^{N_{k_0}}f_nT\otimes y_n}.$$ Suppose on the contrary, for all $k$, we have
 $$\norm{\sum_{n=1}^{N_k}f_nT\otimes y_n}\leq \max\left\{1-\frac{\varepsilon_2}{16},\frac{3}{4}\right\}.$$
 Let $T_{k}=\sum_{n=1}^{N_k}f_nT\otimes y_n$, then $T$ is the limit of $T_{k}$ in the strong operator topology, that is, $\lim_{k\rightarrow\infty}\|T_kx-Tx\|=0$ for all $x\in X$. We claim that $\|T\|\leq\liminf_{k\rightarrow \infty}\|T_{k}\|$. Since $$\|Tx\|\leq \|Tx-T_kx\|+\|T_kx\|,$$ we have $\|Tx\|\leq \liminf_{k\rightarrow\infty}\|T_kx\|$.
 Thus $\|T\|\leq \liminf_{k\rightarrow\infty}\|T_k\|$. Since $\|T_{k}\|\leq \max\left\{1-\varepsilon_2 /16,3/4\right\}<1$ for all $k$, we have $\|T\|<1$, and it is a contradiction.

 Since $X^*$ and $Y$ are separable, there exists a countable dense subset $\mathcal{A}=\{x_n^*\}_{n=1}^\infty$ in $B_{X^\ast}$ and a countable dense subset $\mathcal{B}=\{u_n\}_{n=1}^\infty$ in $B_Y$.  For all  $n=1,2,\cdots$, there exists $x^{*}_{m_n} \in \mathcal{A}$, $m_{n}\in\N_+$ and $u_{l_n}\in \mathcal{B}$, $l_{n}\in\N_+$  such that
  $$\norm{x^{*}_{m_n}-f_{n} T} <\min\left\{\frac{1}{8N_{k_0}},\frac{\varepsilon_2}{64N_{k_0}}\right\} \text{ and } \norm{u_{l_n}-y_n} <\min\left\{\frac{1}{8N_{k_0}},\frac{\varepsilon_2}{64N_{k_0}}\right\}.$$
  Thus \[\norm{\sum_{n=1}^{N_{k_0}}f_{n}T\otimes y_{n} -\sum_{n=1}^{N_{k_0}}x^{*}_{m_n}\otimes u_{l_n}}< \min\left\{\frac{1}{4},\frac{\varepsilon_2}{32}\right\}.\]
  Let $\xi_{k_0}=\norm{\sum_{n=1}^{N_{k_0}} x^{*}_{m_n}\otimes u_{l_n}}$, then we have
  \[\max\left\{1-\frac{3\varepsilon_2}{32},\frac{1}{2}\right\}<\xi_{k_0}< \min\left\{\frac{9}{4},2+\frac{\varepsilon_2}{16}\right\}.\]
 We will show that the countable set
  \[\left\{\frac{2\sum_{i=1}^{m} g_{i}^{\ast}\otimes s_{i}}{\norm{\sum_{i=1}^{m} g_{i}^{\ast}\otimes s_{i}}}: m\in \mathbb{N}_+, g_{i}^{\ast}\in\mathcal{A}\ (1\leq i\leq m), s_{i}\in\mathcal{B}\ (1\leq i\leq m)\right\}\] is a set of BCP points for $B(X,Y)$.

By triangular inequality and a regular computation, we have
\begin{align}
    &\quad\norm{T - \frac{2\sum_{n=1}^{N_{k_0}} x^{*}_{m_n}\otimes u_{l_n}}{\norm{\sum_{n=1}^{N_{k_0}} x^{*}_{m_n}\otimes u_{l_n}}}}\notag\\
    &=\norm{\sum_{n=1}^{N_{k_0}} f_{n} T \otimes y_{n} -\sum_{n=1}^{N_{k_0}} \frac{2 }{\xi_{k_0}} x^{*}_{m_n}\otimes u_{l_n}+\sum_{n=N_{k_0}+1}^{\infty} f_{n} T\otimes y_{n}}\notag\\
    &\leq \norm{\frac{2}{\xi_{k_0}}\left(\sum_{n=1}^{N_{k_0}} f_{n} T\otimes y_{n} -\sum_{n=1}^{N_{k_0}} x^{*}_{m_n}\otimes u_{l_n} \right) }\notag\\
    &\quad+\norm{\left(1-\frac{2}{\xi_{k_0}} \right)\sum_{n=1}^{N_{k_0}} f_{n} T\otimes y_{n}+\sum_{n=N_{k_0}+1}^{\infty} f_{n} T\otimes y_{n} } \notag\\
    &\leq \norm{\frac{2}{\xi_{k_0}}\left(\sum_{n=1}^{N_{k_0}} f_{n} T\otimes y_{n} -\sum_{n=1}^{N_{k_0}} x^{*}_{m_n}\otimes u_{l_n} \right)}\notag\\
    &\quad+\left(2-\varepsilon+\varepsilon_1\right)\max\left\{\left|1-\frac{2}{\xi_{k_0}}\right|,1\right\} \label{ieq1}\\
    &<\frac{\varepsilon_2}{8\xi_{k_0}}+\left(2-\varepsilon+\varepsilon_1\right)
    \left(\frac{2}{\xi_{k_0}}-1\right)\label{ieq3}\\
    &=2-\varepsilon+\varepsilon_1+(2-\varepsilon+\varepsilon_1)
    \left(\frac{2}{\xi_{k_0}}-2\right)+\frac{\varepsilon_2}{8\xi_{k_0}}\notag\\
    &< 2-\varepsilon+\varepsilon_1+(8-8\xi_{k_0})+\frac{1}{4} \varepsilon_2\notag\\
    &< 2-\varepsilon+\varepsilon_1+\frac{3}{4}\varepsilon_2+\frac{1}{4} \varepsilon_2\notag\\
    &<2-\varepsilon+\sigma \label{ieq2}\\
    &< 2\notag,
\end{align}
where inequality (\ref{ieq1}) is obtained by Lemma \ref{lem3.4}. In (\ref{ieq3}), without loss of generality, we can always assume $\xi_{k_0}\leq1$. And inequality (\ref{ieq2}) shows that $B(X,Y)$ has the $(2,\varepsilon-\sigma)$-UBCP.
\end{proof}

\begin{thm}\label{thm3.13}
  Let $X$ be a Banach space with $X^\ast$ separable and $Y$ be a separable Banach space with the $(2-\varepsilon)$-UBAP for any $\varepsilon>0$. Then for all $1-\varepsilon/2 < \alpha \leq 1$, the renormed space $Z_\alpha=\left(B(X,Y), \|\cdot\|_\alpha \right)$ defined in (\ref{renorm}) has the $(2\alpha,2\alpha+\varepsilon-2-\sigma)$-UBCP for all $0<\sigma<2\alpha+\varepsilon-2$.
  \end{thm}
  \begin{proof}
  We only need to prove the case when $1-\varepsilon/2 < \alpha < 1$ as the case $\alpha=1$ has been proved in Theorem \ref{thm3.12}. By Lemma \ref{decomposable}, for all $0<\varepsilon_1<\sigma/2$, there exists a block unconditional Schauder frame $\{(y_n,f_n)\}_{n\in\mathbb{N}} \subseteq Y\times Y^\ast$ such that $$y=\sum_{n=1}^{\infty}f_n(y) y_n=\sum_{n=1}^{\infty}(f_n\otimes y_n)(y)$$ for all $y\in Y$ and there exists a subsequence $\{N_k\}_{k=1}^{\infty}$ of $\N_+$ such that $$\sup_{N\in\N,\theta_k=\pm1}\left\|\sum_{k=1}^{N}\theta_k \sum_{n=N_k+1}^{N_{k+1}}f_n\otimes y_n\right\|\leq 2-\varepsilon+\varepsilon_1 \text{ and } \sup_{N\in\mathbb{N}_+}\left\|\sum_{n=1}^{N}f_n \otimes y_n\right\|\leq
  2-\varepsilon+\varepsilon_1.$$

  For all $T \in B(X,Y)$ with $\|T\|_\alpha=1$, we have $\|T\| \geq 1$ and $\|T\|_{Q(X,Y)} \leq 1$. Since for all $x \in X$, we have
    \[Tx=\sum_{n=1}^{\infty}\left(f_n T\otimes y_n\right) (x).\]
  The following existence of $k_0$ can be proved by the same way as in Theorem \ref{thm3.12}. For all $0<\varepsilon_2< \sigma/4$, there exists $k_0$ such that
   \[\max\left\{\|T\|-\frac{\varepsilon_2}{16},\frac{3}{4}\right\}< \norm{\sum_{n=1}^{N_{k_0}}f_nT\otimes y_n} \leq (2-\varepsilon+\varepsilon_1)\|T\|.\]

   Since $X^*$ and $Y$ are separable, there exists a countable dense subset $\mathcal{A}=\{x_n^*\}_{n=1}^\infty$ in $B_{X^\ast}$ and a countable dense subset $\mathcal{B}=\{u_n\}_{n=1}^\infty$ in $B_Y$.  For all $n=1,2,\cdots$, there exists $x^{*}_{m_n} \in \mathcal{A}$, $m_{n}\in\N_+$ and $u_{l_n}\in \mathcal{B}$, $l_{n}\in\N_+$  such that
   $$\norm{x^{*}_{m_n}-f_{n} T} <\min\left\{\frac{1}{8N_{k_0}},\frac{\varepsilon_2}{64N_{k_0}}\right\} \text{ and } \norm{u_{l_n}-y_n} <\min\left\{\frac{1}{8N_{k_0}},\frac{\varepsilon_2}{64N_{k_0}}\right\}.$$
   Thus \[\norm{\sum_{n=1}^{N_{k_0}}f_{n}T\otimes y_{n} -\sum_{n=1}^{N_{k_0}}x^{*}_{m_n}\otimes u_{l_n}}< \min\left\{\frac{1}{4},\frac{\varepsilon_2}{32}\right\}.\]
   Let $\xi_{k_0}=\norm{\sum_{n=1}^{N_{k_0}} x^{*}_{m_n}\otimes u_{l_n}}$, then we have
   \[1-\frac{3\varepsilon_2}{32}\leq\max\left\{\|T\|-\frac{3\varepsilon_2}{32},\frac{1}{2}\right\}<\xi_{k_0}< \min\left\{\frac{17}{4},2\|T\|+\frac{\varepsilon_2}{16}\right\}.\]
  We will show that the countable set
   \[\left\{\frac{2\sum_{i=1}^{m} g_{i}^{\ast}\otimes s_{i}}{\norm{\sum_{i=1}^{m} g_{i}^{\ast}\otimes s_{i}}}: m\in \mathbb{N}_+, g_{i}^{\ast}\in\mathcal{A}\ (1\leq i\leq m), s_{i}\in\mathcal{B}\ (1\leq i\leq m)\right\}\] is a set of BCP points for $Z_\alpha$ ($1-\varepsilon/2 < \alpha \leq 1$).

 By triangular inequality and a regular computation, we have
 \begin{align}
     &\quad\norm{T - \frac{2\sum_{n=1}^{N_{k_0}} x^{*}_{m_n}\otimes u_{l_n}}{\norm{\sum_{n=1}^{N_{k_0}} x^{*}_{m_n}\otimes u_{l_n}}}}_\alpha\notag\\
     &=\alpha\norm{T - \frac{2\sum_{n=1}^{N_{k_0}} x^{*}_{m_n}\otimes u_{l_n}}{\norm{\sum_{n=1}^{N_{k_0}} x^{*}_{m_n}\otimes u_{l_n}}}}+(1-\alpha)\norm{T - \frac{2\sum_{n=1}^{N_{k_0}} x^{*}_{m_n}\otimes u_{l_n}}{\norm{\sum_{n=1}^{N_{k_0}} x^{*}_{m_n}\otimes u_{l_n}}}}_{Q(X,Y)}\notag\\
     &=\alpha\norm{\sum_{n=1}^{N_{k_0}} f_{n} T \otimes y_{n} -\sum_{n=1}^{N_{k_0}} \frac{2}{\xi_{k_0}} x^{*}_{m_n}\otimes u_{l_n}+\sum_{n=N_{k_0}+1}^{\infty} f_{n} T\otimes y_{n}}\notag\\
     &\quad+(1-\alpha)\norm{T - \frac{2\sum_{n=1}^{N_{k_0}} x^{*}_{m_n}\otimes u_{l_n}}{\xi_{k_0}}}_{Q(X,Y)}\notag\\
     &\leq\alpha\norm{\frac{2}{\xi_{k_0}}\left(\sum_{n=1}^{N_{k_0}} f_{n} T\otimes y_{n} -\sum_{n=1}^{N_{k_0}} x^{*}_{m_n}\otimes u_{l_n} \right)}\notag\\
     &\quad+\alpha\norm{\left(1-\frac{2}{\xi_{k_0}}\right)\sum_{n=1}^{N_{k_0}} f_{n} T \otimes y_{n}+\sum_{n=N_{k_0}+1}^{\infty} f_{n} T\otimes y_{n}}\notag\\
     &\quad+(1-\alpha)\norm{T - \frac{2\sum_{n=1}^{N_{k_0}} x^{*}_{m_n}\otimes u_{l_n}}{\xi_{k_0}}}_{Q(X,Y)}\notag\\
     &\leq \alpha\norm{\frac{2}{\xi_{k_0}}\left(\sum_{n=1}^{N_{k_0}} f_{n} T\otimes y_{n} -\sum_{n=1}^{N_{k_0}} x^{*}_{m_n}\otimes u_{l_n} \right)}\notag\\
     &\quad+\alpha\left(2-\varepsilon+\varepsilon_1\right)\max\left
     \{\left|1-\frac{2}{\xi_{k_0}}\right|,1\right\}\|T\|+(1-\alpha)\|T\|_{Q(X,Y)} \label{ieq4}\\
     &< \frac{\varepsilon_2}{8\xi_{k_0}}+\left(2-\varepsilon+\varepsilon_1\right)
     \left(\frac{2}{\xi_{k_0}}-1\right)\alpha\|T\|+(1-\alpha)\|T\|_{Q(X,Y)}\label{ieq5}\\
     &\leq \frac{\varepsilon_2}{8\xi_{k_0}}+\left(2-\varepsilon+\varepsilon_1\right)\left(\frac{2}{\xi_{k_0}}-1\right)\notag\\
     &=2-\varepsilon+\varepsilon_1+\frac{\varepsilon_2}{8\xi_{k_0}}+(2-\varepsilon+\varepsilon_1)
     \left(\frac{2}{\xi_{k_0}}-2\right)\notag\\
     &< 2-\varepsilon+\frac{\sigma}{2}+\frac{\sigma}{8}+\frac{3\varepsilon_2}{4}\notag\\
     &< 2-\varepsilon+\sigma \label{ieq6}\\
     &= 2\alpha-\left(2\alpha+\varepsilon-2-\sigma\right)\notag\\
     &< 2\alpha \notag\\
     &=\norm{\frac{2\sum_{n=1}^{N_{k_0}} x^{*}_{m_n}\otimes u_{l_n}}{\norm{\sum_{n=1}^{N_{k_0}} x^{*}_{m_n}\otimes u_{l_n}}}}_\alpha, \notag
   \end{align}
where inequality (\ref{ieq4}) is obtained by Lemma \ref{lem3.4}. In inequality (\ref{ieq5}), without loss of generality, we can always assume $\xi_{k_0}\leq1$. And inequality (\ref{ieq6}) shows that $Z_\alpha$ ($1-\varepsilon/2 < \alpha \leq 1$) has the $(2\alpha,2\alpha+\varepsilon-2-\sigma)$-UBCP.
\end{proof}

Thus we have the following result.

\begin{cor}\label{main1}
Let $X$ be a Banach space with $X^\ast$ separable and $Y$ be a Banach space with an $(2-\varepsilon)$-unconditional basis for any $\varepsilon>0$. Then for all $1-\varepsilon/2<\alpha\leq1$, the renormed space $Z_\alpha=(B(X,Y),\|\cdot\|_\alpha)$ has the $(2\alpha, 2\alpha+\varepsilon-2-\sigma)$-UBCP for all $0<\sigma<2\alpha+\varepsilon-2$.
\end{cor}

\begin{re}\label{rem3.14}
In the proof of Theorem \ref{thm3.12} and Theorem \ref{thm3.13}, we need to control \[\norm{\left(1-\frac{2}{\xi} \right)\sum_{n=1}^{N_{k_0}} f_{n} T\otimes y_{n}+\sum_{n=N_{k_0}+1}^{\infty} f_{n} T\otimes y_{n} }=\norm{\left(id_Y-\frac{2}{\xi}\sum_{n=1}^{N_{k_0}}f_n \otimes y_n\right) \circ T }.\]
\end{re}

In \cite{J1972}, Johnson showed that every separable Banach space with the CBAP can be renormed to have the commuting MAP. In \cite{C1991}, Casazza and Kalton showed that a separable Banach space with the MAP has the commuting MAP. These two results show that the CBAP is an isomorphic equivalent form of the MAP. Recall that \cite{C1991} a separable Banach space has the MAP if and only if there exists $\rho>0$ and an approximating sequence $\{S_n\}_{n=1}^\infty$ with $\lim_n\|id_X+\rho S_n\|=1+\rho$. There is also a reverse metric approximation property which implies the CBAP and defined as follows \cite{C2001,C1991}.
\begin{de}\label{RMAP}
  Let $X$ be a separable Banach space with an approximating sequence
  $\{S_n\}_{n=1}^\infty$, then $X$ has the reverse metric approximation property (RMAP for short) if $\lim_n\|id_X-S_n\|=1$. Let $\rho>0$, we say $X$ has the $\rho$-RMAP if $\lim_n\|id_X-\rho S_n\|=1$.
\end{de}

Clearly, the $2$-RMAP is equivalent to the UMAP. It was shown in \cite{C1991} that if $X$ has the $\rho$-RMAP for some $\rho>0$ then $X$ has the RMAP. We call the approximating sequence $\{S_n\}_{n=1}^\infty$ in Definition \ref{RMAP} a reverse metric approximating sequence. By Remark \ref{rem3.14}, we get the following corollary.
\begin{cor}\label{cor 3.6}
  Let $X$ be a Banach space with $X^\ast$ separable, $Y$ be a separable Banach space with the $\lambda$-RMAP and $\{S_n\}_{n=1}^\infty$ be the $\lambda$-reverse metric approximating sequence. If $1< \lambda\leq 2$, then $B(X,Y)$ has the $(2,\sigma)$-UBCP for all $0<\sigma<\lambda-1$.
\end{cor}
\begin{proof}
  Since $\{S_n\}_{n=1}^\infty$ is a $\lambda$-reverse metric approximating sequence, then \[\limsup_{n}\|S_n\|\leq \frac{2}{\lambda}<\infty.\] Thus we have $\lambda\leq 2/\limsup_{n}\|S_n\| \leq 2$ and $Y$ has the BAP. By Lemma \ref{decomposable}, for all $\varepsilon_1>0$, there exists a Schauder frame $\{(y_n,f_n)\}_{n\in\mathbb{N}} \subseteq Y\times Y^\ast$ such that $$y=\sum_{n=1}^{\infty}f_n(y) y_n=\sum_{n=1}^{\infty}(f_n\otimes y_n)(y)$$ for all $y\in Y$ and$$\sup_{N\in\mathbb{N}_+}\left\|\sum_{n=1}^{N}f_n \otimes y_n\right\|\leq
  \limsup_{n}\|S_n\|+\varepsilon_1.$$
  For all $T \in B(X,Y)$ with $\|T\|=1$, let $T_{k}=\sum_{n=1}^{N_k}f_nT\otimes y_n$, then
  $$1=\|T\| \leq  \liminf_{k}\|T_k\| \leq \limsup_{k}\|T_k\| \leq \limsup_{k}\|S_k\|\|T\|=\limsup_{k}\|S_k\|.$$
 Let $\xi_{k}=\norm{\sum_{n=1}^{N_{k}} x^{*}_{m_n}\otimes u_{l_n}}$. Let $k_0$ be large enough and $\xi_{k_0}=\norm{\sum_{n=1}^{N_{k_0}} x^{*}_{m_n}\otimes u_{l_n}}$ be chosen the same as that in Theorem \ref{thm3.12}.
 Thus for any $\varepsilon_2>0$, we have
   \[\max\left\{\|T\|-\frac{3\varepsilon_2}{32},\frac{1}{2}\right\}< \xi_{k_0}<\min\left\{\frac{17}{4},\limsup_{k}\|S_k\|\|T\|+\frac{\varepsilon_2}{16}\right\}.\]
  By passing to a subsequence, we can assume $\lim_k \xi_k=\beta$, then
  \[\max\left\{\|T\|-\frac{3\varepsilon_2}{32},\frac{1}{2}\right\}\leq \beta\leq\min\left\{\frac{17}{4},\limsup_{k}\|S_k\|\|T\|+\frac{\varepsilon_2}{16}\right\}.\]
  Since $Y$ has the $\lambda$-RMAP, we have $\lambda\|S_k\|\leq2$ which implies $\lambda\|S_kT\|\leq2$. Thus $\lambda\beta\leq2$.

  By Remark \ref{rem3.14}, it is sufficient to show that for some $\sigma>0$ the following inequality holds \[\norm{\left(id_Y-\frac{2}{\xi_{k_0}}\sum_{n=1}^{N_{k_0}}f_n \otimes y_n\right) \circ T } < 2-\sigma.\]
  Actually we have
  \[
  \begin{aligned}
    &\quad\norm{\left(id_Y-\frac{2}{\xi_{k_0}}\sum_{n=1}^{N_{k_0}}f_n \otimes y_n\right) \circ T }\\
    &\leq\norm{\left(id_Y-\lambda\sum_{n=1}^{N_{k_0}}f_n \otimes y_n\right) \circ T }+\left|\lambda-\frac{2}{\xi_{k_0}}\right|\norm{\sum_{n=1}^{N_{k_0}}f_n T \otimes y_n}\\
    &\leq 1+\left|\lambda-\frac{2}{\lim_{k}\xi_k}\right|\norm{\sum_{n=1}^{N_{k_0}}f_n T \otimes y_n}+\left|\frac{2}{\lim_{k}\xi_k}-\frac{2}{\xi_{k_0}}\right|\norm{\sum_{n=1}^{N_{k_0}}f_n T \otimes y_n}\\
    &\leq 1+\left|\lambda-\frac{2}{\beta}\right|\|T_{k_0}\|+o(1)\\
    &\leq 1+\left|\lambda-\frac{2}{\beta}\right|\xi_{k_0}+o(1)\\
    &\leq 1+\left|\lambda-\frac{2}{\beta}\right|\beta+o(1)\\
    &=1+\left|\lambda\beta-2\right|+o(1)\\
    &<2,
  \end{aligned}
  \]
  where $o(1)$ can be arbitrarily small when $\varepsilon_1,\varepsilon_2$ small enough. Then the proof presented in Theorem \ref{thm3.12} can be applied with minor modifications.
\end{proof}

By Corollary \ref{cor 3.6}, we obtain the following corollary.

\begin{cor}
  Let $X$ and $Y$ be Banach spaces with $X^\ast$ and $Y$ separable. If $Y$ has the $\lambda$-RMAP for some $\lambda>1$, then $B(X,Y)$ has the UBCP.
\end{cor}

Also, it is true for the renormed spaces.

\begin{cor}\label{cor 3.8}
   Let $X$ be a Banach space with $X^\ast$ separable, $Y$ be a separable Banach space with the $\lambda$-RMAP and $\{S_n\}_{n=1}^\infty$ be the $\lambda$-reverse metric approximating sequence. If $\lambda$ and $(1+\sqrt{17})/8 <\alpha \leq 1$ satisfy
   \[ \frac{1}{\alpha} <\lambda<\frac{4\alpha-1}{\limsup_n\|S_n\|},\] then the renormed space $Z_\alpha=(B(X,Y),\|\cdot\|_\alpha)$ has the $(2\alpha,\sigma)$-UBCP for all \[0<\sigma<\min\{\alpha\lambda-1, 4\alpha-1-\lambda\limsup_n\|S_n\|\}.\]
\end{cor}
\begin{proof}
  If $T \in Z_\alpha$ with $\|T\|_\alpha=1$, then $1 \leq \|T\| \leq \alpha^{-1}$.
  By Remark \ref{rem3.14} and Corollary \ref{cor 3.6}, it is sufficient to show that for some $\sigma>0$ the following inequality holds \[\norm{\left(id_Y-\frac{2}{\xi}\sum_{n=1}^{N_{k_0}}f_n \otimes y_n\right) \circ T }_\alpha < 2\alpha-\sigma.\]
  Actually we have
  \[
  \begin{aligned}
    &\quad\norm{\left(id_Y-\frac{2}{\xi}\sum_{n=1}^{N_{k_0}}f_n \otimes y_n\right) \circ T }_\alpha\\
    &= \alpha\norm{\left(id_Y-\frac{2}{\xi}\sum_{n=1}^{N_{k_0}}f_n \otimes y_n\right) \circ T }+(1-\alpha)\norm{\left(id_Y-\frac{2}{\xi}\sum_{n=1}^{N_{k_0}}f_n \otimes y_n\right) \circ T }_{Q(X,Y)}\\
    &=\alpha\norm{\left(id_Y-\frac{2}{\xi}\sum_{n=1}^{N_{k_0}}f_n \otimes y_n\right) \circ T }+(1-\alpha)\|T \|_{Q(X,Y)}\\
    &\leq \alpha\|T\|+\alpha\left|\lambda-\frac{2}{\beta}\right|\beta+(1-\alpha)\|T \|_{Q(X,Y)}+o(1)\\
    &=1+\alpha\left|\lambda-\frac{2}{\beta}\right|\beta+o(1)\\
    &=1+\left|\alpha\lambda\beta-2\alpha\right|+o(1)\\
    &<2\alpha,
  \end{aligned}
  \]
where $o(1)$ can be arbitrarily small. Then the proof presented in Theorem \ref{thm3.13} can be applied with minor modifications.

\end{proof}

Note that the inequality (\ref{ieq1}) and inequality (\ref{ieq3}) hold if the upper bound of the reflection of approximating operator $id_Y-2S_n$ can be controlled.

\begin{cor}\label{cor 3.9}
  Let $X$ and $Y$ be Banach spaces with $X^\ast$ and $Y$ separable. If $Y$ has an approximating sequence $\{S_n\}_{n=1}^\infty$ such that $\lim_n\|id_Y-2S_n\|=\lambda$ for some $\lambda>0$, then for all \[\frac{\lambda}{4}+\frac{\limsup_n\|S_n\|}{2} < \alpha \leq 1,\] the renormed space $Z_\alpha=\left(B(X,Y), \|\cdot\|_\alpha \right)$ has the $(2\alpha,\sigma)$-UBCP for all \[0<\sigma<4\alpha-2\limsup_{n}\|S_n\|-\lambda.\]
\end{cor}
\begin{proof}
    By Remark \ref{rem3.14} and Corollary \ref{cor 3.8}, we have \begin{align*}
    &\quad\norm{\left(id_Y-\frac{2}{\xi_{k_0}}\sum_{n=1}^{N_{k_0}}f_n \otimes y_n\right) \circ T }_\alpha\\
    &\leq \alpha\norm{\left(id_Y-2\sum_{n=1}^{N_{k_0}}f_n \otimes y_n\right) \circ T }+\alpha\left|2-\frac{2}{\xi_{k_0}}\right|\norm{\sum_{n=1}^{N_{k_0}}f_n T \otimes y_n}+(1-\alpha)\|T\|_{Q(X,Y)}\\
    &\leq \alpha\norm{\left(id_Y-2\sum_{n=1}^{N_{k_0}}f_n \otimes y_n\right) \circ T }+\alpha\left|2-\frac{2}{\lim_{k}\xi_k}\right|\norm{\sum_{n=1}^{N_{k_0}}f_n T \otimes y_n}\\
    & \quad +\alpha\left|\frac{2}{\lim_{k}\xi_k}-\frac{2}{\xi_{k_0}}\right|\norm{\sum_{n=1}^{N_{k_0}}f_n T \otimes y_n}+(1-\alpha)\|T\|_{Q(X,Y)}\\
    &\leq\alpha\lambda\|T\|+(1-\alpha)\|T\|_{Q(X,Y)}+\alpha\left|2-\frac{2}{\beta}\right|\|T_{k_0}\|+o(1)\\
    &= 1+\alpha(\lambda-1)\|T\|+\alpha\left|2-\frac{2}{\beta}\right|\|T_{k_0}\|+o(1)\\
    &\leq 1+\alpha(\lambda-1)\|T\|+\alpha\left|2-\frac{2}{\beta}\right|\xi_{k_0}+o(1)\\
    &\leq \lambda+\alpha\left|2-\frac{2}{\beta}\right|\beta+o(1)\\
    &=\lambda+\left|2\alpha\beta-2\alpha\right|+o(1)\\
    &<2\alpha,
    \end{align*}
where $o(1)$ can be arbitrarily small. Then the proof presented in Theorem \ref{thm3.13} can be applied with minor modifications.
\end{proof}

\begin{cor}
     Let $X$ and $Y$ be Banach spaces with $X^\ast$ and $Y$ separable. If $Y$ has an approximating sequence $\{S_n\}_{n=1}^\infty$ such that $\lim_n\|id_Y-2S_n\| < 3/2$, then  $B(X,Y)$ has the UBCP.
\end{cor}
\begin{proof}
    Since $\lim_n\|id_Y-2 S_n\|<3/2$, we have $\limsup_n\|S_n\| \leq 5/4$. Thus $4-2\limsup_n\|S_n\|-\lim_n\|id_Y-2 S_n\| > 4-5/2-3/2=0$ and  $B(X,Y)$ has the UBCP by Corollary \ref{cor 3.9}.
\end{proof}

Let $X$ be a separable Banach space with the BAP, then $X$ has a Schauder frame $\{(x_n,f_n)\}_{n=1}^\infty \subseteq X\times X^*$. Thus for all $x\in X$ we have $x=\sum_{n=1}^{\infty} f_n(x) x_n$.
Then for all $x \in X$ and $T \in B(X,Y)$, we have
\[Tx=T\left(\sum_{n=1}^{\infty}f_n(x)x_n\right)=\sum_{n=1}^{\infty}\left(f_n\otimes T x_n\right)(x).\] Therefore, we obtain the symmetric version of Theorem \ref{thm3.13} and its following corollaries.

\begin{thm}\label{thm3.19}
  Let $X$ and $Y$ be Banach spaces with $X^\ast$ and $Y$ separable. If $X$ or $Y$ has the $(2-\varepsilon)$-UBAP for any $\varepsilon>0$, then for all $1-\varepsilon/2 < \alpha \leq 1$, the renormed space $Z_\alpha=\left(B(X,Y), \|\cdot\|_\alpha \right)$ has the $(2\alpha,2\alpha+\varepsilon-2-\sigma)$-UBCP for all $0<\sigma<2\alpha+\varepsilon-2$.
\end{thm}
\begin{proof}
The case $Y$ with the $(2-\varepsilon)$-UBAP for any $\varepsilon>0$ has been proved in Theorem \ref{thm3.13}.
If $X$ has the $(2-\varepsilon)$-UBAP for any $\varepsilon>0$, then by Lemma \ref{decomposable}, for all $0<\varepsilon_1<\sigma/2$, there exists a block unconditional Schauder frame $\{(x_n,f_n)\}_{n\in\mathbb{N}} \subseteq X\times X^\ast$ such that $$x=\sum_{n=1}^{\infty}f_n(x) x_n=\sum_{n=1}^{\infty}(f_n\otimes x_n)(x)$$ for all $x\in X$ and there exists a subsequence $\{N_k\}_{k=1}^{\infty}$ of $\N_+$ such that $$\sup_{N\in\N,\theta_k=\pm1}\left\|\sum_{k=1}^{N}\theta_k \sum_{n=N_k+1}^{N_{k+1}}f_n\otimes x_n\right\|\leq 2-\varepsilon+\varepsilon_1 \text{ and } \sup_{N\in\mathbb{N}_+}\left\|\sum_{n=1}^{N}f_n \otimes x_n\right\|\leq
2-\varepsilon+\varepsilon_1.$$

For all $T \in B(X,Y)$ with $\|T\|_\alpha=1$ and for all $x \in X$, we have
\[Tx=T\left(\sum_{n=1}^{\infty}f_n(x)x_n\right)=\sum_{n=1}^{\infty}\left(f_n\otimes T x_n\right)(x).\]
By similar proof with Theorem \ref{thm3.13}, we can show that the countable set
   \[\left\{\frac{2\sum_{i=1}^{m} g_{i}^{\ast}\otimes s_{i}}{\norm{\sum_{i=1}^{m} g_{i}^{\ast}\otimes s_{i}}}: m\in \mathbb{N}_+, g_{i}^{\ast}\in\mathcal{A}\ (1\leq i\leq m), s_{i}\in\mathcal{B}\ (1\leq i\leq m)\right\}\] is a set of BCP points for $Z_\alpha$ ($1-\varepsilon/2 < \alpha \leq 1$).

 Actually we have
 \begin{align}
     &\quad\norm{T - \frac{2\sum_{n=1}^{N_{k_0}} x^{*}_{m_n}\otimes u_{l_n}}{\norm{\sum_{n=1}^{N_{k_0}} x^{*}_{m_n}\otimes u_{l_n}}}}_\alpha\notag\\
     &=\alpha\norm{\sum_{n=1}^{N_{k_0}} f_{n} \otimes T x_{n} -\sum_{n=1}^{N_{k_0}} \frac{2}{\xi_{k_0}} x^{*}_{m_n}\otimes u_{l_n}+\sum_{n=N_{k_0}+1}^{\infty} f_{n} \otimes T x_{n}}\notag\\
     &\quad+(1-\alpha)\norm{T - \frac{2\sum_{n=1}^{N_{k_0}} x^{*}_{m_n}\otimes u_{l_n}}{\xi_{k_0}}}_{Q(X,Y)}\notag\\
     &\leq\alpha\norm{\frac{2}{\xi_{k_0}}\left(\sum_{n=1}^{N_{k_0}} f_{n} \otimes T x_{n} -\sum_{n=1}^{N_{k_0}} x^{*}_{m_n}\otimes u_{l_n} \right)}\notag\\
     &\quad+\alpha\norm{\left(1-\frac{2}{\xi_{k_0}}\right)\sum_{n=1}^{N_{k_0}} f_{n} \otimes T x_{n}+\sum_{n=N_{k_0}+1}^{\infty} f_{n} \otimes T x_{n}}\notag\\
     &\quad+(1-\alpha)\norm{T - \frac{2\sum_{n=1}^{N_{k_0}} x^{*}_{m_n}\otimes u_{l_n}}{\xi_{k_0}}}_{Q(X,Y)}\notag\\
     &\leq \alpha\norm{\frac{2}{\xi_{k_0}}\left(\sum_{n=1}^{N_{k_0}} f_{n} \otimes T x_{n} -\sum_{n=1}^{N_{k_0}} x^{*}_{m_n}\otimes u_{l_n} \right)}\notag\\
     &\quad+\alpha\left(2-\varepsilon+\varepsilon_1\right)\max\left
     \{\left|1-\frac{2}{\xi_{k_0}}\right|,1\right\}\|T\|+(1-\alpha)\|T\|_{Q(X,Y)}\notag\\
     &< \frac{\varepsilon_2}{8\xi_{k_0}}+\left(2-\varepsilon+\varepsilon_1\right)
     \left(\frac{2}{\xi_{k_0}}-1\right)\alpha\|T\|+(1-\alpha)\|T\|_{Q(X,Y)}\notag\\
     &\leq \frac{\varepsilon_2}{8\xi_{k_0}}+\left(2-\varepsilon+\varepsilon_1\right)\left(\frac{2}{\xi_{k_0}}-1\right)\notag\\
     &< 2-\varepsilon+\frac{\sigma}{2}+\frac{\sigma}{8}+\frac{3\varepsilon_2}{4}\notag\\
     &< 2-\varepsilon+\sigma \notag\\
     &= 2\alpha-\left(2\alpha+\varepsilon-2-\sigma\right)\notag\\
     &< 2\alpha.\notag
   \end{align}
Thus the renormed space $Z_\alpha$ ($1-\varepsilon/2 < \alpha \leq 1$) has the $(2\alpha,2\alpha+\varepsilon-2-\sigma)$-UBCP.
\end{proof}
\begin{cor}
   Let $X$ and $Y$ be Banach spaces with $X^\ast$ and $Y$ separable, and one of them has the $\lambda$-RMAP. Let $\{S_n\}_{n=1}^\infty$ be the $\lambda$-reverse metric approximating sequence. If $\lambda$ and $(1+\sqrt{17})/8 <\alpha \leq 1$ satisfy
   \[ \frac{1}{\alpha} <\lambda<\frac{4\alpha-1}{\limsup_n\|S_n\|},\] then the renormed space $Z_\alpha=(B(X,Y),\|\cdot\|_\alpha)$ has the $(2\alpha,\sigma)$-UBCP for all \[0<\sigma<\min\{\alpha\lambda-1, 4\alpha-1-\lambda\limsup_n\|S_n\|\}.\]
\end{cor}
\begin{cor}
  Let $X$ and $Y$ be Banach spaces with $X^\ast$ and $Y$ separable. If $X$ or $Y$ has the $\varepsilon$-RMAP for some $\varepsilon>1$, then $B(X,Y)$ has the UBCP.
\end{cor}

\begin{cor}
  Let $X$ and $Y$ be Banach spaces with $X^\ast$ and $Y$ separable. If $X$ or $Y$ has an approximating sequence $\{S_n\}_{n=1}^\infty$ such that $\lim_n\|id-2S_n\|=\lambda$ for some $\lambda>0$, then for all \[\frac{\lambda}{4}+\frac{\limsup_n\|S_n\|}{2} < \alpha \leq 1,\]the renormed space $Z_\alpha=\left(B(X,Y), \|\cdot\|_\alpha \right)$ has the $(2\alpha,\sigma)$-UBCP for all \[0<\sigma<4\alpha-2\limsup_{n}\|S_n\|-\lambda.\]
\end{cor}
\begin{cor}
     Let $X$ and $Y$ be Banach spaces with $X^\ast$ and $Y$ separable.  If $X$ or $Y$ has an approximating sequence $\{S_n\}_{n=1}^\infty$ such that $\lim_n\|id-2S_n\| < 3/2$, then  $B(X,Y)$ has the UBCP.
\end{cor}

The mutual symmetry between the two theorems suggests that we need to minimize the influence of the domain and range and instead shift our focus to operators. Here we introduce the notion of (unconditional) bounded approximation property of operators.
\begin{de}\label{defn3.22}(\cite{LSZ2023})
  Let $X$ and $Y$ be separable Banach spaces. Let $\lambda >0$ and $S \in B(X,Y)$. Then $S$ is said to have the $\lambda$-bounded approximation property ($\lambda$-BAP) if there exists  $\left\{S_{n}\right\}_{n \in \N_+}\subseteq B(X,Y)$ with
\[\sup_{n \in \N_+}\|S_{n}\| \leq \lambda\] such that both $\dim(S_{n}X)<\infty$ for all $n \in \N_+$ and $\lim_{n}\|S_{n}x-S x\|=0$ for all $x \in X$. We say $S$ has the bounded approximation property (BAP) if there exists a $\lambda>0$ such that $S$ has the $\lambda$-bounded approximation property.
\end{de}

\begin{de}\label{defn3.23}
  Let $X$ and $Y$ be separable Banach spaces. Let $\lambda >0$ and $S \in B(X,Y)$. Then $S$ is said to have the $\lambda$-unconditional bounded approximation property ($\lambda$-UBAP) if for all $\varepsilon>0$ there exists  $\left\{R_{n}=S_n-S_{n-1}\right\}_{n \in \N_+}\subseteq B(X,Y)$ and $S_0=0$ with
\[\sup_{N \in \N_+, \varepsilon_n=\pm1}\left\|\sum_{n=1}^{N}\varepsilon_n R_{n}\right\| \leq \lambda+\varepsilon\] such that both $\dim(S_{n}X)<\infty$ for all $n \in \N_+$ and $\lim_{n}\|S_{n}x-S x\|=0$ for all $x \in X$.
\end{de}

\begin{de}
  Let $X$ be a separable Banach space, $Y$ be a Banach space and $T \in B(X,Y)$. Then $T$ is said to be linear decomposable if there exists $\left\{(y_{n},f_{n})\right\}_{n\in\N}\subseteq Y \times X^* $ such that $T x=\sum_{n=1}^{\infty} f_{n}(x) y_{n}$ for all $x \in X$.  And $T$ is said to be block unconditional decomposable if there exists a subsequence
  $\{N_k\}_{k=1}^\infty$ of $\N_+$ such that $$\sup_{N\in\mathbb{N_+},\theta_k=\pm1}\left\|\sum_{k=1}^{N}\theta_k \sum_{n=N_k+1}^{N_{k+1}}f_n\otimes y_n\right\| < \infty.$$
\end{de}

It was proved in \cite{LSZ2023} that $T$ has the BAP if and only if $T$ is linear decomposable for all $T \in B(X,Y)$.

The method of proof in Lemma \ref{decomposable} carries over to the relationship between the $\lambda$-UBAP and block unconditional decomposition for all $T \in B(X,Y)$.
\begin{lem}
  Let $T\in B(X,Y)$ with the $\lambda$-UBAP and $\varepsilon > 0$. Then there exists a (block unconditional) decomposition  $\{(y_n,f_n)\}_{n\in\mathbb{N}+} \subseteq Y\times X^\ast$ such that for all $x\in X$ we have
  $$Tx=\sum_{n=1}^{\infty}f_n(x) y_n.$$  And there exists a subsequence
  $\{N_k\}_{k=1}^{\infty}$ of $\N_+$ such that
  $$\sup_{N\in\mathbb{N}_+,\theta_k=\pm1}\left\|\sum_{k=1}^{N}\theta_k \sum_{n=N_k+1}^{N_{k+1}}f_n \otimes y_n\right\|\leq \lambda+\varepsilon  \text{ and } \sup_{N\in\mathbb{N}_+}\left\|\sum_{n=1}^{N}f_n \otimes y_n\right\|\leq \lambda+\varepsilon.$$
  In addition, we can assume $\|y_n\| = 1$ and $\|f_n\| \leq 1$  for all $n \in \N_+$.
\end{lem}

Let $T \in B(X,Y)$ with the UBAP, then define
$$\textbf{UBAP}(T)=\inf\{\lambda: T \text{ has the } \lambda\text{-UBAP}\}.$$ Thus if $X$ has the UBAP then
\[\textbf{UBAP}(id_X)=\textbf{UBAP}(X).\]
And if either $\textbf{UBAP}(X)=\lambda$ or $\textbf{UBAP}(Y)=\lambda$, then for all $T \in B(X,Y)$, we have $\textbf{UBAP}\left(T/\|T\|\right) \leq \lambda$. Therefore we have the following result.
\begin{thm}\label{thm3.26}
  Let $X$ and $Y$ be Banach spaces with $X^\ast$ and $Y$ separable. For any $\varepsilon>0$, define \[\mathcal{Z}_\varepsilon=\left\{T \in B(X,Y): \textbf{UBAP}\left(\frac{T}{\|T\|}\right) \leq 2-\varepsilon \right\}.\]
  Then for all $1-\varepsilon/2 < \alpha \leq 1$, $\left(\mathcal{Z}_\varepsilon, \|\cdot\|_\alpha \right)$ has the $(2\alpha,2\alpha+\varepsilon-2-\sigma)$-UBCP for all $0<\sigma<2\alpha+\varepsilon-2$.
\end{thm}

Next we consider the case when either $X$ or $Y$ has the UMAP. The UMAP is associated with the lattice structure of compact operator space $K(X)$ and $B(X)$. Casazza and Kalton \cite{C1991} introduced the notion of u-ideal which is a generalization of M-ideal \cite{AE1972,HW1993} and it is well known that if $X$ is an M-ideal in $Y$ then $X$ is a u-ideal.

\begin{de}[\cite{C1991}]
  Let $X$ be a subspace of a Banach space $Y$. Then $X$ is a u-ideal in $Y$ if there is a projection $P$ from $Y^*$ onto $X^{\bot}$ (the kernel of $P$) such that $\|id-2P\|=1$.
\end{de}

The following theorem demonstrates the important relationship between the UMAP and u-ideals.

\begin{thm}[\cite{C1991}]
Let $X$ be a separable reflexive Banach space with the approximation property. Then $X$ has the UMAP if and only if $K(X)$ is a u-ideal in $B(X)$.
\end{thm}

Also, u-ideals can be characterized by the ball intersection property.

\begin{thm}[\cite{LA2007}]\label{BIP}
Let $X$ be a closed subspace of a Banach space $Y$ and let $y \in Y \backslash X$ and $Z=\operatorname{span}(X,\{y\})$. The following statements are equivalent.
\begin{enumerate}
\item[(1)] $X$ is a u-ideal in $Z$.
\item[(2)] $X^{\perp \perp} \bigcap_{x \in X} B_{Z^{* *}}(y+x,\|y-x\|) \neq \emptyset$.
\item[(3)] $X \cap \bigcap_{i=1}^n B_Z\left(y+x_i,\left\|y-x_i\right\|+\varepsilon\right) \neq \emptyset$ for every finite collection $\left(x_i\right)_{i=1}^n \subseteq X$ and $\varepsilon>0$.
\item[(4)] $X \cap \bigcap_{i=1}^3 B_Z\left(y+x_i,\left\|y-x_i\right\|+\varepsilon\right) \neq \emptyset$ for every collection of three points $\left(x_i\right)_{i=1}^3 \subseteq X$ and $\varepsilon>0$.
\end{enumerate}
\end{thm}


Thus we get the following corollary.
\begin{cor}\label{u-ideal and BIP}
  Let $X$ and $Y$ be separable reflexive Banach spaces with the approximation property.
  \begin{enumerate}
    \item[(1)]   If either $K(X)$ is a u-ideal in $B(X)$ or $K(Y)$ is a u-ideal in $B(Y)$,  then for all $1/2 < \alpha \leq 1$, the renormed space $Z_\alpha=\left(B(X,Y), \|\cdot\|_\alpha \right)$ has the $(2\alpha,2\alpha-1-\sigma)$-UBCP for all $0<\sigma<2\alpha-1$.
    \item[(2)] If either $K(X)$ or $K(Y)$ has the ball intersection property defined in Theorem \ref{BIP}, then for all $1/2 < \alpha \leq 1$, the renormed space $Z_\alpha=\left(B(X,Y), \|\cdot\|_\alpha \right)$ has the $(2\alpha,2\alpha-1-\sigma)$-UBCP for all $0<\sigma<2\alpha-1$.
  \end{enumerate}
\end{cor}
For classsical sequence space, by Corollary \ref{u-ideal and BIP}, we have the following results.
\begin{cor}\label{cor3.31}
  $B(\ell_p(F_n),\ell_q(H_n))$ has the UBCP for all finite-dimensional Banach space sequences $\{F_n\}_{n=1}^\infty$ and $\{H_n\}_{n=1}^\infty$ where $1< p,q <\infty$.
\end{cor}

\begin{cor}
  Let $C_p$ denote the space $\ell_p(F_n)$ where $F_n$ is a sequence of finite-dimensional Banach spaces dense in the Banach-Mazur sense in the collection of all finite-dimensional Banach spaces. Then $B(C_p,C_q)$ has the UBCP for $1< p,q <\infty$.
\end{cor}


\begin{exa}
  In \cite{S1987}, Szarek constructed a Banach space $\ell_2(W_n)$ which has the $1$-UFDD thus has the $1$-UBAP but fails to have a basis, where $W_n$ is a sequence of normed spaces. By Corollary \ref{cor3.31}, if $W_n$ is a sequence of finite-dimensional Banach spaces, then $B(\ell_2(W_n))$ has the UBCP.
\end{exa}

In \cite{LGK1996}, Godefroy, Kalton and Li gave a description of the subspace with the UMAP of $L^1[0,1]$.

\begin{thm}[\cite{LGK1996}]
  Let $X$ be a closed subspace of $L^1[0,1]$ with the approximation property. The following assertions are equivalent.
  \begin{enumerate}
    \item[(1)] The unit ball of $X$ is compact locally convex in the topology of convergence in measure.
    \item[(2)] $X$ has the UMAP and the $1$-strong Schur property.
    \item[(3)] For any $\varepsilon>0$, there is a quotient $Y_\varepsilon$ of $c_0$ such that
    $d(X, Y_{\varepsilon}^\ast) < 1 + \varepsilon$.
  \end{enumerate}
\end{thm}

Therefore, we get the following sufficient condition for $B(X,Y)$ to have the UBCP.
\begin{cor}
  Let $X$ be a Banach space with $X^\ast$ separable and $Y$ be a closed subspace of $L^1[0,1]$ with the approximation property. If the unit ball of $Y$ is compact locally convex in the topology of convergence in measure, then $B(X,Y)$ has the UBCP.
\end{cor}
Following Proposition 4.5 in \cite{LGK1996} which qualifies the UMAP, we obtain two spaces of bounded linear operators whose range spaces are concrete natural spaces have the UBCP.
\begin{exa}
 Let $Y$ be a separable predual of a von Neumann algebra. Then $Y$ has the UMAP if and only if $Y$ is isometric to $\ell_1(N(H_n))$, where $N(H_n)$ is a space of nuclear operators on the finite-dimensional Hilbert space $H_n$ for every $n$. Thus if $X$ is a Banach space with $X^*$ separable and $Y$ is isometric to $\ell_1(N(H_n))$, then $B(X,Y)$ has the UBCP.
\end{exa}

\begin{exa}
  There exist infinite subsets $\Lambda$ of $\Z$ such that $L_{\Lambda}^1(\mathbb{T})$ has the UMAP. Thus if $X$ is a Banach space with $X^\ast$ separable, then $B(X, L_{\Lambda}^1(\mathbb{T}))$ has the UBCP.
\end{exa}


\begin{thebibliography}{50}




\bibitem{AE1972} E.M. Alfsen, E.G. Effros, Structure in real Banach spaces, Parts I and II, Ann. of Math. 96 (1972) 98–173.

\bibitem{AG2023} A. Avil\'{e}s, G. Mart\'{\i}nez-Cervantes, A. Rueda Zocz, A renorming characterisation of Banach spaces containing $\ell_{1}(k)$, Publ. Math. 67 (2023) 601--609.

\bibitem{B1932} S. Banach, Th\'{e}orie des Op\'{e}rations Lin\'{e}arires, Warszawa, 1932.

\bibitem{BHLS2025} Q. Bao, D. Han, R. Liu, J. Shen, Linearization of Lipschitz framings for Banach spaces, Expo. Math. 43 (2025) 1--18.

\bibitem{BLS2025} Q. Bao, R. Liu, J. Shen, The ball-covering property of non-commutative spaces of operators on Banach spaces, Banach J. Math. Anal. 19 (2025) 1--20.

\bibitem{BF2024} K. Beanland, D. Freeman, Shrinking Schauder frames and their associated bases, Constr. Approx. 60 (2024) 443--461.

\bibitem{BFL2015} K. Beanland, D. Freeman, R. Liu, Upper and lower estimates for Schauder frames and atomic decompositions, Fund. Math. 231 (2015) 161--188.

\bibitem{C2001} P.G. Casazza,  Approximation properties, in: W.B. Johnson, J. Lindenstrauss (Eds.), Handbook of the Geometry of Banach Spaces, Amsterdam, North-Holland, 2001, pp. 271--316.

\bibitem{ST2008} P.G. Casazza, S.J. Dilworth, E. Odell, T. Schlumprecht, A. Zs\'{a}k, Coefficient quantization for frames in Banach spaces, J. Math. Anal. Appl. 348 (2008) 66--86.  

\bibitem{CHL} P.G. Casazza, D. Han, D.R. Larson, Frames for Banach spaces, Contemp. Math. 247 (1999) 149--182.

\bibitem{C1991} P.G. Casazza, N.J. Kalton, Notes on approximation properties in separable Banach spaces, in: P.F.X. M\"{u}ller, W. Schachermayer (Eds.), Proc. Conf. Geometry of Banach Spaces, Strobl, 1989, in: London Math. Soc. Lecture Note Ser., vol. 158, Cambridge Univ. Press, 1990, pp. 49--63.


\bibitem{C2006} L. Cheng, Ball-covering property of Banach spaces, Israel J. Math. 156 (2006) 111--123.

\bibitem{CCL2008} L. Cheng, Q. Cheng, X. Liu, Ball-covering property of Banach spaces that is not preserved under linear isomorphisms, Sci. China Ser. A: Math. 51 (2008) 143--147.


\bibitem{CKZ2020} L. Cheng, M. Kato, W. Zhang, A survey of ball-covering property of Banach spaces, in: The Mathematical Legacy of Victor Lomonosov--Operator Theory, Adv. Anal. Geom. 2 (2020) 67--84.

\bibitem{CLL2010} L. Cheng, Z. Luo, X. Liu, W. Zhang, Several remarks on ball-coverings of normed spaces, Acta Math.   Sin. (Engl. Ser.) 26 (2010) 1667--1672.

\bibitem{CSZ2009} L. Cheng, H. Shi, W. Zhang, Every Banach spaces with a $w^*$-separable dual has a $1+\varepsilon$-equivalent norm with the ball covering property, Sci. China Ser. A: Math. 52 (2009) 1869--1874.

\bibitem{CWZ2011} L. Cheng, B. Wang, W. Zhang, Y. Zhou, Some geometric and topological properties of Banach spaces via ball coverings, J. Math. Anal. Appl. 377 (2011) 874--880.

\bibitem{CLL2023} S. Ciaci, J. Langemets, A. Lissitsin, A characterization of Banach spaces containing $\ell_{1}(k)$ via ball-covering properties, Israel J. Math. 253 (2023) 359--379.

\bibitem{E1973} P. Enﬂo, A counterexample to the approximation problem in Banach spaces, Acta Math. 130 (1973) 309--317.


\bibitem{FR2016} V.P. Fonf, M. Rubin, A reconstruction theorem for locally convex metrizable spaces, homeomorphism groups without small sets, semigroups of non-shrinking functions of a normed space, Topology Appl. 210 (2016)  97--132.

\bibitem{FZ32009} V.P. Fonf, C. Zanco, Covering spheres of Banach spaces by balls, Math. Ann. 344 (2009) 939--945.

\bibitem{ST2014} D. Freeman, E. Odell, T. Schlumprecht, A. Zs\'{a}k, Unconditional structures of translates for $L_p(\mathbb{R}^d)$, Israel J. Math. 203 (2014) 189--209.

\bibitem{GK1997} G. Godefroy, N.J. Kalton, Approximating sequences and bidual projections, Q. J. Math. 48 (1997)  179--202.

\bibitem{G1955} A. Grothendieck, Produits tensoriels topologiques et espaces nucl\'{e}aires, Mem. Amer. Math.
Soc. 16 (1955) 193--200.


\bibitem{ADV2024} D. Han, Q. Hu, D.R. Larson, R. Liu, Dilations for operator-valued quantum measures, Adv. Math. 438 (2024) 1--33.

\bibitem{JFA2018} D. Han, D.R. Larson, R. Liu, Dilations of operator-valued measures with bounded p-variations and framings on Banach spaces, J. Funct. Anal. 274 (2018) 1466--1490.

\bibitem{MAMS2014} D. Han, D.R. Larson, B. Liu, R. Liu, Operator-valued measures, dilations, and the theory of
frames, Mem. Amer. Math. Soc. 229 (2014).

\bibitem{JFA2014} D. Han, D.R. Larson, B. Liu, R. Liu, Dilations for systems of imprimitivity acting on Banach spaces,  J. Funct. Anal. 266 (2014) 6914--6937.

\bibitem{HW1993} P. Harmand, D. Werner, W. Werner, M-Ideals in Banach Spaces and Banach Algebras, Lecture Notes in Math., vol. 1547, Springer-Verlag, Berlin, 1993.

\bibitem{J1972} W.B. Johnson, A complementably universal conjugate Banach space and its relation to the approximation
problem, Israel J. Math. 13 (1972) 301--310.

\bibitem{JRZ1971} W.B. Johnson, H.P. Rosenthal, M. Zippin, On bases, finite dimensional decompositions and weaker structures in Banach spaces, Israel J. Math. 9 (1971) 488--506.

\bibitem{LT2025} N. Lev, A. Tselishchev, There are no unconditional Schauder frames of translates in $L^p(\mathbb{R})$, $1\leqslant p\leqslant 2$, Adv. Math. 460 (2025) 1--11.

\bibitem{LGK1996} D. Li, G. Godefroy, N.J. Kalton, On subspaces of $L^1$ which embed into $\ell_1$, J. Reine Angew.  Math. 471 (1996) 43--75.

\bibitem{LA2007} V. Lima, \AA. Lima, A three-ball intersection property for u-ideals, J. Funct. Anal. 252 (2007)  220--232.

\bibitem{L2010} R. Liu, On shrinking and boundedly complete Schauder frames of Banach spaces, J. Math. Anal. Appl. 365 (2010) 385--398.

\bibitem{LLLZ2022} M. Liu, R. Liu, J. Lu, B. Zheng, Ball covering property from commutative function spaces to non-commutative spaces of operators, J. Funct. Anal. 283 (2022) 1--15.

\bibitem{LR2016} R. Liu, Z. Ruan, Cb-frames for operator spaces, J. Funct. Anal. 270 (2016) 4280--4296.

\bibitem{LSZ2023} R. Liu, J. Shen, B. Zheng, Operators with the Lipschitz bounded approximation property, Sci. China Math. 66 (2023) 1545--1554.

\bibitem{LZ2010} R. Liu, B. Zheng, A characterization of Schauder frames which are near-Schauder bases, J. Fourier Anal. Appl. 16 (2010) 791--803. 


\bibitem{LZ2020} Z. Luo, B. Zheng, Stability of the ball-covering property, Studia Math. 250 (2020) 19--34.

\bibitem{LZ2021} Z. Luo, B. Zheng, The strong and uniform ball covering properties, J. Math. Anal. Appl. 499 (2021) 1--15.

\bibitem{M1998} R.E. Megginson, An Introduction to Banach Space Theory, Graduate Texts in Mathematics, vol. 183, Springer, 1998.

\bibitem{OSS2011} E. Odell, B. Sari, T. Schlumprecht, B. Zheng, Systems formed by translates of one element in $L_p(\mathbb{R})$, Trans. Amer. Math. Soc. 363 (2011) 6505--6529.

\bibitem{P1971} A. Pe{\l}czy\'{n}ski, Any separable Banach space with the bounded approximation property is a complemented subspace of a Banach space with a basis, Studia Math. 40 (1971) 239--243.

\bibitem{S2021} S. Shang, The ball-covering property on dual spaces and Banach sequence spaces, Acta Math. Sci. Ser. B (Engl. Ed.) 41 (2021) 461--474.


\bibitem{SC2015} S. Shang, Y. Cui, Locally $2$-uniform convexity and ball-covering property in Banach space, Banach J.   Math. Anal. 9 (2015) 42--53.

\bibitem{SC2018} S. Shang, Y. Cui, Dentable point and ball-covering property in Banach spaces, J. Convex Anal. 25 (2018) 1045--1058.

\bibitem{S1987} S. Szarek, A Banach space without a basis which has the bounded approximation property, Acta Math. 159 (1987) 81--98.


\bibitem{Z2012} W. Zhang, Characterizations of universal ﬁnite representability and B-convexity of Banach spaces via ball coverings, Acta Math. Sin. (Engl. Ser.) 28 (2012) 1369--1374.


\end{thebibliography}
\end{document}